\documentclass[12pt,reqno]{amsart}

\usepackage{mathdots}

\usepackage{color}

\usepackage{pdfsync}

\usepackage{amssymb}

\usepackage{mathrsfs}

\DeclareMathAlphabet{\mathpzc}{OT1}{pzc}{m}{it}

\usepackage[all]{xy}

\usepackage{tikz}
\usetikzlibrary{arrows,decorations.pathmorphing,decorations.pathreplacing,positioning,shapes.geometric,shapes.misc,decorations.markings,decorations.fractals,calc,patterns}

\usepackage{float}

\usepackage[bottom]{footmisc}

\entrymodifiers={+!!<0pt,\fontdimen22\textfont2>}

\def\BC{\mathbb{C}}

\def\BZ{\mathbb{Z}}

\def\sA{\mathsf{A}}
\def\sB{\mathsf{B}}
\def\sC{\mathsf{C}}

\def\sR{\mathsf{R}}

\def\add{\operatorname{add}}

\def\adots{\mathinner{\mkern1mu\raise1.0pt\vbox{\kern7.0pt\hbox{.}}\mkern2mu\raise4.0pt\hbox{.}\mkern2mu\raise7.0pt\hbox{.}\mkern1mu}}

\def\ast{{\textstyle *}}

\def\corad{\operatorname{corad}}

\def\dddots{\mathinner{\mkern1mu\raise10.0pt\vbox{\kern7.0pt\hbox{.}}\mkern2mu\raise5.3pt\hbox{.}\mkern2mu\raise1.0pt\hbox{.}\mkern1mu}}
\def\dddotssmall{\mathinner{\mkern1mu\raise7.0pt\vbox{\kern7.0pt\hbox{.}}\mkern-1mu\raise4pt\hbox{.}\mkern-1mu\raise1.0pt\hbox{.}\mkern1mu}}

\def\dim{\operatorname{dim}}

\def\dual{\operatorname{D}}

\def\Ext{\operatorname{Ext}}

\def\fl{\mathsf{fl}}

\def\Gr{\operatorname{Gr}}

\def\Hom{\operatorname{Hom}}
\def\id{\operatorname{id}}
\def\Image{\operatorname{Im}}

\def\ind{\operatorname{ind}}

\def\K{\operatorname{K}}

\def\Ker{\operatorname{Ker}}

\def\mod{\mathsf{mod}}
\def\Mod{\mathsf{Mod}}

\def\obj{\operatorname{obj}}
\def\opp{\operatorname{op}}

\def\rad{\operatorname{rad}}

\def\SL2{\operatorname{SL}_2}

\numberwithin{equation}{section}

\newtheorem{Lemma}{Lemma}[section]
\newtheorem{Theorem}[Lemma]{Theorem}
\newtheorem{Proposition}[Lemma]{Proposition}
\newtheorem{Corollary}[Lemma]{Corollary}

\theoremstyle{definition}
\newtheorem{Definition}[Lemma]{Definition}
\newtheorem{Setup}[Lemma]{Setup}

\newtheorem{Remark}[Lemma]{Remark}

\newtheorem{bfhpg}[Lemma]{}               
\newtheorem*{bfhpg*}{}

\begin{document}

\setlength{\parindent}{0pt}
\setlength{\parskip}{7pt}

\title[Generalised friezes and modified Caldero-Chapoton]{Generalised
friezes and a modified Caldero-Chapoton map depending on a rigid object}

\author{Thorsten Holm}
\address{Institut f\"{u}r Algebra, Zahlentheorie und Diskrete
Mathematik, Fa\-kul\-t\"at f\"ur Mathematik und Physik, Leibniz
Universit\"{a}t Hannover, Welfengarten 1, 30167 Hannover, Germany}
\email{holm@math.uni-hannover.de}
\urladdr{http://www.iazd.uni-hannover.de/\~{ }tholm}

\author{Peter J\o rgensen}
\address{School of Mathematics and Statistics,
Newcastle University, Newcastle upon Tyne NE1 7RU, United Kingdom}
\email{peter.jorgensen@ncl.ac.uk}
\urladdr{http://www.staff.ncl.ac.uk/peter.jorgensen}




\begin{abstract} 

  The (usual) Caldero-Chapoton map is a map from the set of objects of a
  category to a Laurent polynomial ring over the integers.  In the case
  of a cluster category, it maps ``reachable'' indecomposable objects
  to the corresponding cluster variables in a cluster algebra.  This
  formalises the idea that the cluster category is a
  ``categorification'' of the cluster algebra.

  The definition of the Caldero-Chapoton map requires the category to
  be $2$-Calabi-Yau, and the map depends on a cluster tilting object
  in the category.

  We study a modified version of the Caldero-Chapoton map which only
  requires the category to have a Serre functor, and only depends on
  a rigid object in the category.

  It is well-known that the usual Caldero-Chapoton map gives rise to
  so-called friezes, for instance Conway-Coxeter friezes.  We show
  that the modified Caldero-Chapoton map gives rise to what we call
  generalised friezes, and that for cluster categories of Dynkin type
  $A$, it recovers the generalised friezes introduced by combinatorial
  means in \cite{BHJ}.

\end{abstract}

\maketitle

\setcounter{section}{-1}
\section{Introduction}
\label{sec:introduction}

The (usual) Caldero-Chapoton map is an important object in the
homological part of cluster theory, see \cite[3.1]{CC}.  Among other
things, it gives rise to so-called friezes.  In particular,
Conway-Coxeter friezes can be recovered like this, see \cite[sec.\
5]{CC}.

This paper studies a modified version of the Caldero-Chapoton map.  We
show that it gives rise to what we call generalised friezes.  In
particular, the generalised friezes which were introduced by
combinatorial means in \cite{BHJ} can be recovered like this.

\subsection{Background}
We first explain what the usual Caldero-Chapoton map is.  If $Q$ is
a finite quiver without loops and $2$-cycles, then there is a cluster
algebra $A( Q )$ and a cluster category $\sC( Q )$ of type $Q$, see
\cite{BMRRT} and \cite{FZ}.


The algebra $A( Q )$ and the category $\sC( Q )$ are linked by the
Caldero-Chapoton map $\rho_T$ which depends on a cluster tilting
object $T \in \sC( Q )$, see \cite{CC}, \cite{CK}, \cite{CK2},
\cite{Palu}, and \cite{Palu2}.  It is a map from the set of objects of
$\sC( Q )$ to a Laurent polynomial ring over $\BZ$.  Its image
generates $A( Q )$ which embeds into Laurent polynomials.  Indeed,
$\rho_T$ maps ``reachable'' indecomposable objects to cluster
variables and formalises the idea that the cluster category is a
``categorification'' of the cluster algebra.

Note that $\rho_T$ can actually be defined on any $2$-Calabi-Yau
category $\sC$ with a cluster tilting object $T$, and that one of its
good properties is that it is a so-called frieze, see \cite[def.\
1.1]{AD}, \cite[prop.\ 3.10]{CC}, and \cite[theorem]{DG}.  This means
that it is a map from the set of objects of $\sC$ to a ring,
satisfying $\rho_T( c_1 \oplus c_2 ) = \rho_T( c_1 )\rho_T( c_2 )$,
such that if $\tau c \rightarrow b \rightarrow c$ is an
Auslander-Reiten (AR) triangle in $\sC$ then
\begin{equation}
\label{equ:fr}
  \rho_T( \tau c )\rho_T( c ) - \rho_T( b ) = 1.
\end{equation}

Moreover, since $\rho_T$ has values in a Laurent polynomial ring over
$\BZ$, setting all the variables equal to $1$ gives a frieze with
values in $\BZ$.

A classic case of this arises for $\sC( A_n )$, the cluster category
of Dynkin type $A_n$.  For example, the AR quiver of $\sC( A_7 )$ is
shown in Figure \ref{fig:AR_quiver}.  The quiver is $\BZ A_7$ modulo a
glide reflection, so the two dotted line segments in the figure should
be identified with opposite orientations.
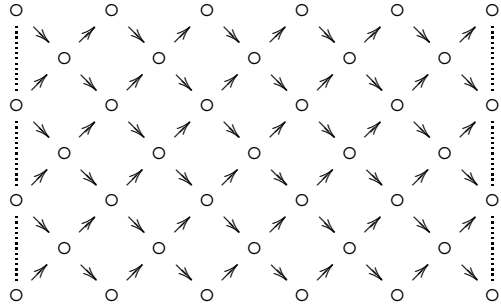
\begin{figure}
\[
  \vcenter{
  \xymatrix @-0.5pc @!0 {
    \circ \ar[dr] \ar@{.}[dd] && \circ \ar[dr] && \circ \ar[dr] && \circ \ar[dr] && \circ \ar[dr] && \circ \ar@{.}[dd] \\
    & \circ \ar[dr] \ar[ur] && \circ \ar[dr] \ar[ur] && \circ \ar[dr] \ar[ur] && \circ \ar[dr] \ar[ur] && \circ \ar[dr] \ar[ur] & \\
    \circ \ar[dr] \ar[ur] \ar@{.}[dd] && \circ \ar[dr] \ar[ur] && \circ \ar[dr] \ar[ur] && \circ \ar[dr] \ar[ur] && \circ \ar[dr] \ar[ur] && \circ \ar@{.}[dd] \\
    & \circ \ar[dr] \ar[ur] && \circ \ar[dr] \ar[ur] && \circ \ar[dr] \ar[ur] && \circ \ar[dr] \ar[ur] && \circ \ar[dr] \ar[ur] & \\
    \circ \ar[dr] \ar[ur] \ar@{.}[dd] && \circ \ar[dr] \ar[ur] && \circ \ar[dr] \ar[ur] && \circ \ar[dr] \ar[ur] && \circ \ar[dr] \ar[ur] && \circ \ar@{.}[dd] \\
    & \circ \ar[dr] \ar[ur] && \circ \ar[dr] \ar[ur] && \circ \ar[dr] \ar[ur] && \circ \ar[dr] \ar[ur] && \circ \ar[dr] \ar[ur] & \\
    \circ \ar[ur] && \circ \ar[ur] && \circ \ar[ur] && \circ \ar[ur] && \circ \ar[ur] && \circ \\
                        }
          }
\]
\caption{The Auslander-Reiten quiver of the cluster category $\sC( A_7 )$.}
\label{fig:AR_quiver}
\end{figure}
Figure \ref{fig:An} shows a $\BZ$-valued frieze, obtained as
described, by giving its values on the indecomposable objects of $\sC(
A_7 )$.
\begin{figure}
\[
  \vcenter{
  \xymatrix @-0.5pc @!0 {
    4 \ar@{.}[dd] && 4 && 1 && 2 && 2 && 4 \ar@{.}[dd] \\
    & 15 && 3 && 1 && 3 && 7 & \\
    11 \ar@{.}[dd] && 11 && 2 && 1 && 10 && 5 \ar@{.}[dd] \\
    & 8 && 7 && 1 && 3 && 7 & \\
    5 \ar@{.}[dd] && 5 && 3 && 2 && 2 && 11 \ar@{.}[dd] \\
    & 3 && 2 && 5 && 1 && 3 & \\
    4 && 1 && 3 && 2 && 1 && 4 \\
                        }
          }
\]
\caption{A frieze on the cluster category $\sC( A_7 )$.  This is also known as a Conway-Coxeter frieze.}
\label{fig:An}
\end{figure}
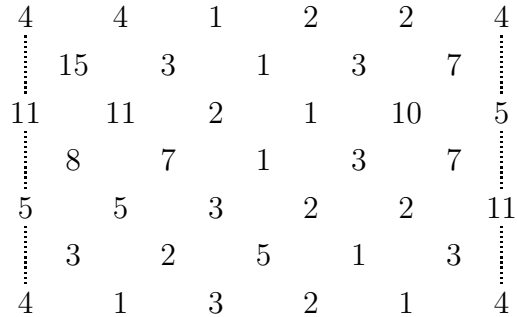
Observe that Equation \eqref{equ:fr} implies that if
\begin{equation}
\label{equ:diamond}
\vcenter{
  \xymatrix @-2.2pc {
    & \beta & \\ \alpha & & \delta \\ & \gamma &
                    }
        }
\end{equation}
is a ``diamond'' in the frieze, then $\alpha\delta - \beta\gamma = 1$.
This is because such a diamond corresponds to a ``mesh'' in the AR
quiver, hence to an AR triangle.

Friezes like this are known as Conway-Coxeter friezes and were studied
long before cluster theory, see \cite{CC1} and \cite{CC2}.  They can
also be defined by combinatorial means based on triangulations of
polygons, see \cite{BCI}.

\subsection{This paper}
We will study a modified version of the Caldero-Chapoton map which
does not require the category $\sC$ to be $2$-Calabi-Yau, but merely
that it has a Serre functor.  Moreover, it does not depend on a
cluster tilting object $T$, but on a rigid object $R$, that is, an
object satisfying the weaker condition $\sC( R,\Sigma R ) = 0$.  Note
that $\sC( -,- )$ is shorthand for the $\Hom$-functor in $\sC$.

To be precise, let $\BC$ be the field of complex numbers, $\sC$ an
essentially small $\BC$-linear $\Hom$-finite triangulated category with
split idempotents and a Serre functor, $R \in \sC$ a rigid object, and
$E = \sC( R,R )$ the endomorphism algebra.  Consider the category
$\Mod\, E$ of $E$-right-modules and the functor
\begin{equation}
\label{equ:ampersand}
  \begin{array}{ccc}
    \sC & \stackrel{G}{\longrightarrow} & \Mod\, E, \\[2mm]
    c   & \longmapsto                   & \sC( R,\Sigma c ).
  \end{array}
\end{equation}
Note that $G$ actually has values in $\mod\, E$, the category of
$E$-modules which are finite dimensional over $\BC$, but we prefer to
view it as having values in $\Mod\, E$ because of a later
generalisation. 

The modified Caldero-Chapoton map determined by $R$ is given by the
following formula.
\[
  \rho_R( c ) = \sum_e \chi\big( \Gr_e( Gc ) \big)
\]
Here $c \in \sC$ is an object, $\Gr_e( Gc )$ is the Grassmannian of
$E$-submodules $M \subseteq Gc$ with $\K_0$-class satisfying $[M] =
e$, and $\chi$ is the Euler characteristic defined by cohomology with
compact support, see \cite[p.\ 93]{F}.  The sum is over $e \in \K_0(
\mod\, E )$.  This gives a map $\rho_R : \obj\, \sC \rightarrow \BZ$.

One of our main results is the following.

{\bf Theorem A. }
{\it
The map $\rho_R : \obj\, \sC \rightarrow \BZ$ is a generalised frieze.
That is,  
\begin{enumerate}

  \item  $\rho_R( c_1 \oplus c_2 ) = \rho_R( c_1 )\rho_R( c_2 )$.

\medskip

  \item  If $\Delta = \tau c \rightarrow b \rightarrow c$ is an AR
  triangle in $\sC$, then the difference $\rho_R( \tau c )\rho_R( c ) -
  \rho_R( b )$ equals $0$ or $1$.

\end{enumerate}
}

In fact, the difference in part (ii) is $0$ or $1$ depending on
whether $G( \Delta )$ is a split short exact sequence or not.  If the
difference in part (ii) were always $1$, then $\rho_R$ would be a
frieze in the earlier sense.

The idea of permitting the difference to be $0$ or $1$ occurred in
\cite{BHJ}, where generalised friezes on $\sC( A_n )$ were introduced
by purely combinatorial means based on higher angulations of polygons;
see paragraph \ref{bfhpg:BHJ} for details.  For example, Figure
\ref{fig:An2} shows the values of such a generalised frieze on the
indecomposable objects of $\sC( A_7 )$.  Note that for each
``diamond'' as in Equation \eqref{equ:diamond} we have $\alpha\delta -
\beta\gamma$ equal to $0$ or $1$.
\begin{figure}
\[
  \vcenter{
  \xymatrix @-0.5pc @!0 {
    3 \ar@{.}[dd] && 2 && 1 && 1 && 2 && 2 \ar@{.}[dd] \\
    & 6 && 2 && 1 && 2 && 4 & \\
    6 \ar@{.}[dd] && 6 && 1 && 1 && 4 && 4 \ar@{.}[dd] \\
    & 6 && 3 && 1 && 2 && 4 & \\
    4 \ar@{.}[dd] && 3 && 2 && 1 && 2 && 6 \ar@{.}[dd] \\
    & 2 && 2 && 2 && 1 && 3 & \\
    2 && 1 && 2 && 1 && 1 && 3 \\
                        }
          }
\]
\caption{A generalised frieze on the cluster category $\sC( A_7 )$, as
introduced in \cite{BHJ}.}
\label{fig:An2}
\end{figure}

It is another main result that the generalised friezes of \cite{BHJ}
can be recovered from the modified Caldero-Chapoton map.

{\bf Theorem B. }
{\it
Let $\sC = \sC( A_n )$ be the cluster category of type $A_n$.  

It follows from \cite{CCS} that a rigid object $R \in \sC$ without
repeated indecomposable summands corresponds to a polygon dissection
of an $( n+3 )$-gon $P$.

By \cite{BHJ} such a polygon dissection defines a generalised
frieze on $\sC$, and this generalised frieze equals $\rho_R$.
}

Note that it is not explicit in \cite{BHJ} that its generalised
friezes are defined on $\sC( A_n )$, but it is established that they
have the requisite periodicity to be so.  Moreover, \cite{BHJ}
requires that $R$ corresponds not just to a polygon dissection of $P$,
but to a higher angulation.  However, this turns out to be an
unnecessary restriction, both for the combinatorial definition in
\cite{BHJ} and for $\rho_R$.

This paper only considers the above version of the Caldero-Chapoton
map with values in $\BZ$.  In the sequel \cite{HJ} we consider a more
elaborate version,
\[
  \rho_R( c ) = \alpha( c ) \sum_e \chi\big( \Gr_e( Gc ) \big)\beta( e ),
\]
where $\alpha$ and $\beta$ have values in a Laurent polynomial ring.
In particular, we will obtain a version of the generalised friezes of
\cite{BHJ} with values in Laurent polynomials.

The paper is organised as follows: Section \ref{sec:modR} gives some
background from representation theory and Section \ref{sec:Gr} shows a
few properties of Grassmannians.  Section \ref{sec:friezes} proves
Theorem A, Section \ref{sec:extension} proves another useful property
of $\rho_R$, and Section \ref{sec:BHJ} proves Theorem B.

Note that Sections \ref{sec:modR} and \ref{sec:Gr} sum up and adapt
some well-known material to our setting.  In these sections we make no
claim to originality.  However, it did not seem feasible to replace
them with references.

\section{Modules over $\sR$}
\label{sec:modR}

This section sums up some items from representation theory.  Most of
them go back to \cite{AusRepDim}, \cite{AusRepI}, \cite{AusRepII}, and
\cite{AR}.

\begin{Setup}
\label{set:blanket}
Throughout, $\BC$ is the field of complex numbers and $\sC$ is an
essentially small $\BC$-linear $\Hom$-finite triangulated category with
split idempotents and Serre functor $S$.  The suspension functor of
$\sC$ is denoted $\Sigma$.

Moreover, $\sR$ is a functorially finite subcategory of $\sC$, closed
under direct sums and summands, which is rigid, that is, $\sC( \sR ,
\Sigma \sR ) = 0$.  Here $\sC( -,- )$ is short for $\Hom_{ \sC }( -,-
)$.
\end{Setup}

\begin{bfhpg}
[The case $\sR = \add\, R$]
\label{bfhpg:rigid}
An important special case is $\sR = \add R$ where $R \in \sC$ is
rigid, that is, $\sC( R,\Sigma R ) = 0$.  Then $\sR$ is automatically
functorially finite, and we have the endomorphism algebra $E = \sC(
R,R )$, the category of $E$-right-modules $\Mod\, E$, and the functor
$G$ from Equation \eqref{equ:ampersand}.  This is the situation from
the introduction.

However, $\sR$ only has the form $\add R$ when it has finitely many
indecomposable objects, and we want to permit infinitely many because
there are nice examples where it is relevant, see e.g.\ \cite[sec.\
6]{JP}.  This requires the following, more general machinery.
\end{bfhpg}

\begin{bfhpg}
[Krull-Schmidt categories]
\label{bfhpg:Krull-Schmidt}
Since $\sC$ is $\BC$-linear $\Hom$-finite with split idempotents, it is
Krull-Schmidt.  So is $\sR$, since it is closed under direct sums
and summands.
We denote the sets of indecomposable objects by $\ind\, \sC$ and
$\ind\, \sR$.
Note that $\sR$ being rigid implies that $\Sigma^{ -1 }( \ind\, \sR )$
and $\ind\, \sR$ are disjoint.
\end{bfhpg}

\begin{bfhpg}
[The category $\Mod\, \sR$]
\label{bfhpg:ModR}
We let $\Mod\, \sR = ( \sR^{ \opp },\Mod\, \BC )$ denote the category of
$\BC$-linear contravariant functors $\sR \rightarrow \Mod\, \BC$.  It is
an abelian category where a sequence $K \rightarrow L \rightarrow M$
is exact if and only if its evaluation at each object of $\sR$ is
exact, see \cite[sec.\ 2]{AusRepI}.

There is a functor
\[
  \begin{array}{ccc}
    \sC & \stackrel{G}{\longrightarrow} & \Mod\, \sR, \\[2mm]
    c   & \longmapsto                   & \sC( -,\Sigma c )|_{ \sR }.
  \end{array}
\]
Note that $G( \sR ) = 0$. 

If $\sR = \add\, R$ where $R$ is a rigid object, and $E = \sC( R,R )$
is the endomorphism algebra, then there is an equivalence
\[
  \begin{array}{ccc}
    \Mod\, \sR & \stackrel{\sim}{\longrightarrow} & \Mod\, E, \\[2mm]
    M          & \longmapsto                      & M( R )
  \end{array}
\]
which identifies the two versions of $G$ given in this
paragraph and Equation \eqref{equ:ampersand}.

Note that $\Mod\, \sR = ( \sR^{ \opp },\Mod\, \BC )$ has the subcategory
$( \sR^{ \opp },\mod\, \BC )$ of $\BC$-linear contravariant functors $\sR
\rightarrow \mod\, \BC$.  It is closed under subobjects and quotients,
so is an abelian subcategory of $\Mod\, \sR$ with exact inclusion
functor.
\end{bfhpg}

\begin{bfhpg}
[Projective objects]
\label{bfhpg:projectives}
An object $r \in \sR$ gives a projective object
\[
  P_r( - ) = \sR( -,r ) = G( \Sigma^{ -1 }r )
\]
in $\Mod\, \sR$.  For an object $M \in \Mod\, \sR$, Yoneda's Lemma
says that there is an isomorphism
\begin{equation}
\label{equ:Yoneda}
  \Hom_{ \Mod\, \sR }( P_r,M ) \rightarrow M( r )
\end{equation}
given by mapping a natural transformation $P_r = \sR( -,r )
\rightarrow M$ to its evaluation on $\id_r$.

If $r \in \ind\, \sR$ then $P_r$ is indecomposable and has a unique
maximal proper subobject, $\rad\, P_r$.  Hence a morphism $M
\rightarrow P_r$ which is not an epimorphism factors through $\rad\,
P_r \hookrightarrow P_r$.  See \cite[sec.\ 2]{AusRepI} and \cite[props.\
2.2 and 2.3]{AusRepII}.
\end{bfhpg}

\begin{bfhpg}
[The category $\mod\, \sR$]
\label{bfhpg:coherents}
An object $M \in \Mod\, \sR$ is called coherent if there is an exact
sequence
\[
  P_{ r_1 } \rightarrow P_{ r_0 } \rightarrow M \rightarrow 0
\]
with $r_0, r_1 \in \sR$.  The full subcategory of
coherent objects is denoted by $\mod\, \sR$.  It is clearly contained
in $( \sR^{ \opp },\mod\, \BC )$.  Since $\sR$ is functorially finite in
$\sC$, the category $\mod\, \sR$ is abelian by \cite[rmk.\ after def.\
2.9]{IY} and the inclusion $\mod\, \sR \hookrightarrow \Mod\, \sR$ is
exact by \cite[sec.\ III.2]{AusRepDim}.
\end{bfhpg}

\begin{bfhpg}
[Dualising variety]
\label{bfhpg:dualising_variety}
Composition with the functor $\dual( - ) = \Hom_{ \BC }( -,\BC )$
gives a duality
\[
  ( \sR^{ \opp },\mod\, \BC ) \rightarrow ( \sR,\mod\, \BC ).
\]
By \cite[props.\ 2.10 and 2.11]{IY} the category $\sR$ is a dualising
variety in the sense of \cite[sec.\ 2]{AR}, so the displayed duality
restricts to a duality
\[
  \mod\, \sR \rightarrow \mod\, \sR^{ \opp }.
\]
\end{bfhpg}

\begin{bfhpg}
[Simple and finite length objects]
\label{bfhpg:fl}
The simple objects of $\Mod\, \sR$ are precisely those of the form
\[
  S_r = P_r / \rad\, P_r
\]
for $r \in \ind\, \sR$,
see \cite[props.\ 2.2 and 2.3]{AusRepII}.  Since $\sR$ is a dualising
variety, $S_r \in \mod\, \sR$ for each $r \in \ind\, \sR$ by
\cite[prop.\ 3.2(c)]{AR}.  As in \cite[(1.4)]{JP} it follows that
$\mod\, \sR$ and $\Mod\, \sR$ have the same simple and the same finite
length objects.  We denote the full subcategory of finite length
objects by $\fl\, \sR$.  It is closed under subobjects and quotients
in $\mod\, \sR$ and in $\Mod\, \sR$, so is abelian and the inclusion
functors $\fl\, \sR \hookrightarrow \mod\, \sR$ and $\fl\, \sR
\hookrightarrow \Mod\, \sR$ are exact.
%
%
%
\end{bfhpg}

\begin{bfhpg}
[$\K$-theory]
\label{bfhpg:K}
It is immediate from paragraph \ref{bfhpg:fl} that $\K_0( \fl\, \sR )$
is a free group on the generators $[ S_r ]$ for $r \in \ind\, \sR$,
where $[-]$ denotes the $\K_0$-class of an object.  If $M \in \fl\,
\sR$ then $M$ has a finite filtration with simple quotients and the
$\K_0$-class $[ M ]$ is the sum of the $\K_0$-classes of the simple
quotients.  For $M' \subseteq M$ this implies that
\begin{equation}
\label{equ:K}
  [ M' ] = [ M ] \Leftrightarrow M' = M
\;\;,\;\;
  [ M' ] = 0     \Leftrightarrow M' = 0.
\end{equation}
\end{bfhpg}

\begin{bfhpg}
[Injective objects]
\label{bfhpg:injectives}
The previous items are left/right symmetric so if $r \in \ind\, \sR$
then $\overline{P}_r = \sR( r,- )$ is indecomposable projective in
$\Mod\, \sR^{ \opp }$ and there is a short exact sequence
\[
  0
  \rightarrow \rad\, \overline{P}_r
  \rightarrow \overline{P}_r
  \rightarrow \overline{S}_r
  \rightarrow 0
\]
in $\Mod\, \sR^{ \opp }$ where $\overline{S}_r$ is simple in $\Mod\,
\sR^{ \opp }$.
The sequence is in $( \sR , \mod\, \BC )$ and dualising it gives
a short exact sequence
\[
  0
  \rightarrow S_r
  \rightarrow I_r
  \rightarrow \corad I_r
  \rightarrow 0
\] 
where
\[
  I_r = \dual\!\sR( r,- ) = \sR( -,Sr )
\]
is indecomposable injective in $\Mod\, \sR$.  A morphism $I_r
\twoheadrightarrow N$ which is not a monomorphism factors through $I_r
\twoheadrightarrow \corad\, I_r$.
%
%
%
%
%
\end{bfhpg}

The next two lemmas follow by standard methods.  We include short
proofs for completeness.  Note that if $\sA$ and $\sB$ are full
subcategories of $\sC$ then $\sA * \sB$ denotes the full subcategory
of objects $x$ appearing in distinguished triangles $a \rightarrow x
\rightarrow b$ with $a \in \sA$, $b \in \sB$.

\begin{Lemma}
\label{lem:5}
\begin{enumerate}

  \item For $M \in \mod\, \sR$ there is $z \in ( \Sigma^{ -1 } \sR ) *
  \sR$ such that $Gz \cong M$.

\medskip

  \item For $z \in ( \Sigma^{ -1 } \sR ) * \sR$ and $c \in \sC$, the
    map 
\[
  \sC( z,c )
  \stackrel{ G( - ) }{ \longrightarrow }
  \Hom_{ \Mod\, \sR }( Gz,Gc )
\]
is surjective.
\end{enumerate}
\end{Lemma}

\begin{proof}
(i) For $M \in \mod\, \sR$ there is an exact sequence $P_{ r_1 }
\rightarrow P_{ r_0 } \rightarrow M \rightarrow 0$ with $r_0, r_1 \in
\sR$.  By Equation \eqref{equ:Yoneda} the first arrow is induced by a
morphism $r_1 \rightarrow r_0$ in $\sR$.  Desuspending and completing
to a distinguished triangle $\Sigma^{ -1 }r_1 \rightarrow \Sigma^{ -1
}r_0 \rightarrow z \rightarrow r_1$ in $\sC$, it is easy to check $M
\cong Gz$.

(ii) For $r \in \sR$, Equation \eqref{equ:Yoneda} gives an isomorphism
$\Hom_{ \Mod\, \sR }( P_r,Gc ) \rightarrow ( Gc )( r )$ which can also
be written $\Hom_{ \Mod\, \sR } \big( G( \Sigma^{ -1 }r ),Gc \big)
\rightarrow \sC( \Sigma^{ -1 }r,c )$.  One checks that its inverse is
$G( - )$ which is hence bijective in this case.

Now let $z \in ( \Sigma^{ -1 }\sR ) * \sR$ be given.  There
is a distinguished triangle $\Sigma^{ -1 }r_1 \rightarrow \Sigma^{ -1
}r_0 \rightarrow z \rightarrow r_1$ which induces an exact sequence
$G( \Sigma^{ -1 }r_1 ) \rightarrow G( \Sigma^{ -1 }r_0 ) \rightarrow
Gz \rightarrow 0$ and a commutative diagram
\[
  \xymatrix @-0.25pc @C=2ex {
    \sC( r_1,c ) \ar[r] \ar[d] & \sC( z,c ) \ar[r] \ar^{ G( - ) }[d] & \sC( \Sigma^{ -1 }r_0,c ) \ar[r] \ar^{ G( - ) }[d] & \sC( \Sigma^{ -1 }r_1,c ) \ar^{ G( - ) }[d] \\
    0 \ar[r] & \Hom_{ \Mod\, \sR } \big( Gz,Gc \big) \ar[r] & \Hom_{ \Mod\, \sR } \big( G ( \Sigma^{ -1 }r_0 ),Gc \big) \ar[r] & \Hom_{ \Mod\, \sR } \big( G ( \Sigma^{ -1 }r_1 ),Gc \big)
                     }
\]
with exact rows.  The first vertical arrow is surjective, and the
third and fourth vertical arrows are bijective by the previous part of
the proof.  The Four Lemma implies that the second vertical arrow is
surjective as claimed.
\end{proof}

Now let
\[
  \Delta = \tau c \rightarrow b \stackrel{ \beta }{ \rightarrow } c
\]
be an AR triangle in $\sC$ whence
\[
  G( \Delta ) = G( \tau c ) \rightarrow Gb \rightarrow Gc
\]
is an exact sequence.

\begin{Lemma}
\label{lem:6a}
\begin{enumerate}

  \item  If $c = \Sigma^{-1}r \in \Sigma^{-1} \ind\, \sR$
         then $G( \Delta ) = 0 \rightarrow \rad\, P_r \rightarrow P_r$.

\medskip

  \item  If $c = r \in \ind\, \sR$
         then $G( \Delta ) = I_r \rightarrow \corad\, I_r \rightarrow 0$.

\medskip

  \item  If $c \not\in \Sigma^{ -1 }( \ind\, \sR ) \cup \ind\, \sR$ 
         then $G( \Delta )$ is a short exact sequence.

\end{enumerate}
\end{Lemma}

\begin{proof}
(i)  Let $c = \Sigma^{ -1 }r$ whence $Gc = P_r$.

Pick a right $\sR$-approximation $r' \stackrel{ \rho' }{ \rightarrow }
\Sigma b$.  It is easy to see that composing with $\Sigma b \stackrel{
  \Sigma \beta }{ \rightarrow } r$ gives a morphism $r' \rightarrow r$
which is almost splitable in the sense of \cite[sec.\ 2]{AusRepII}, so
the row in the following diagram is exact by \cite[cor.\
2.6]{AusRepII}.
\[
  \xymatrix @-0.25pc {
    \sC( -,r' )|_{ \sR } \ar[rr] \ar@{->>}_{ \rho'_{ * }}[dr] & & \sC( -,r )|_{ \sR } \ar^-{ \sigma }[rr] & & S_r \ar[rr] & & 0 \\
    & \sC( -,\Sigma b )|_{ \sR } \ar_{ ( \Sigma \beta )_{ * } }[ur]
                     }
\]
Since $\sigma$ is the canonical epimorphism $P_r \rightarrow S_r$, the
diagram shows $\Image\, ( \Sigma \beta )_{ \ast } = \rad\, P_r$.  This
can also be written $\Image G\beta = \rad\, P_r$.

Finally, $c = \Sigma^{ -1 }r$ implies
\[
  G( \tau c )
  = \sC( -,\Sigma \tau c )|_{ \sR }
  = \sC \big( -,\Sigma ( S \Sigma^{ -1 } )( \Sigma^{ -1 }r ) \big) \big|_{ \sR }
  = \sC( -,S\Sigma^{ -1 }r )|_{ \sR }
  = \dual\!\sC( \Sigma^{ -1 }r,- )|_{ \sR } 
  = 0.
\]

The sequence $G( \Delta )$ is exact, and combining with what we have
shown gives $G( \Delta ) = 0 \rightarrow \rad\, P_r \rightarrow P_r$
as desired. 

(ii)  Apply part (i) to $\sC^{ \opp }$ and $\sR^{ \opp }$ and
dualise.

(iii)  There is a long exact sequence
\[
  \xymatrix {
  G( \Sigma^{ -1 }b ) \ar^-{ G( \Sigma^{ -1 }\beta ) }[rr]
  & & G( \Sigma^{ -1 }c ) \ar[r] 
  & G( \tau c ) \ar[r]
  & Gb \ar^{ G \beta }[r]
  & Gc.
                     }
\]
The first morphism can also be written $\sC( -,b )|_{ \sR } \stackrel{
\beta_{ * } }{ \rightarrow } \sC( -,c )|_{ \sR }$.  It is an
epimorphism when $c \not\in \ind\,\sR$, since $\beta$ is right almost
split. Similarly, the last morphism in the long exact sequence is an
epimorphism when $\Sigma c \not\in \ind\, \sR$, and part (iii) of the
proposition follows.
\end{proof}

\section{Grassmannians}
\label{sec:Gr}

This section adapts some material from \cite{CC}, \cite{CK},
\cite{CK2}, \cite{Palu}, and \cite{Palu2} to our setting.

\begin{Definition}
[Grassmannians] Let $M \in \Mod\, \sR$ and $e \in \K_0( \fl\, \sR )$
be given.  Let $\Gr( M )$ be the Grassmannian of subobjects $M'
\subseteq M$ with finite length, and let $\Gr_e( M ) \subseteq \Gr( M
)$ be the Grassmannian of subobjects $M' \subseteq M$ with finite
length and $[ M' ] = e$.
\end{Definition}

\begin{bfhpg}
[Constructible maps]
\label{bfhpg:constructible_maps}
A morphism $M \stackrel{ j }{ \rightarrow } N$ in $\fl\, \sR$ induces
constructible maps of Grassmannians as follows.
\[
  \begin{array}{ccccccc}
    \Gr( M ) & \rightarrow & \Gr( N ) & , & \Gr( N ) & \rightarrow & \Gr( M )\\[3mm]
    M'       & \mapsto     & jM'      & , & N'       & \mapsto     & j^{ -1 }N'
  \end{array}
\]
See \cite[sec.\ 2.1]{Palu2} for the definitions of constructible sets
and maps.  Note that in particular, the image and the inverse image
under a constructible map of a constructible set are constructible.
\end{bfhpg}

\begin{Setup}
\label{set:Gr}
For the rest of this section $a \rightarrow b \rightarrow c$ are fixed
morphisms in $\sC$.  We assume that applying $G$ gives a short exact
sequence 
\begin{equation}
\label{equ:astast}
  0
  \rightarrow Ga
  \stackrel{i}{\rightarrow} Gb
  \stackrel{p}{\rightarrow} Gc
  \rightarrow 0
\end{equation}
and that $Ga$, $Gb$, $Gc$ have finite length in $\Mod\, \sR$.
\end{Setup}

\begin{Definition}
For $e, f \in K_0( \fl\, \sR )$, there is a constructible subset
\[
  X_{ e,f } = \big\{\, L \in \Gr( Gb )
                      \,\big|\, [ i^{ -1 }L ] = e,
                                [ pL ] = f \,\big\}
            \subseteq \Gr( Gb )
\]
and a morphism
\[
  \begin{array}{ccccc}
    X_{ e,f } & \stackrel{\pi_{ e,f }}\longrightarrow & \Gr_e( Ga ) & \times & \Gr_f( Gc ) \lefteqn{\;\; ,}\\[3mm]
    L & \longmapsto & ( \;\; i^{ -1 }L & , & pL \;\; ) \lefteqn{\;\; .}
  \end{array}
\]
\end{Definition}

\begin{Lemma}
\label{lem:union}
For each $g \in \K_0( \fl\, \sR )$ we have
\[
  \Gr_g( Gb ) = \bigcup_{ e+f = g } X_{ e,f }
\]
where the right hand side is a finite disjoint union.
\end{Lemma}

\begin{proof}
Each $L \in \Gr( Gb )$ is a subobject of $Gb$ so sits in a short exact
sequence $0 \rightarrow i^{ -1 }L \rightarrow L \rightarrow pL
\rightarrow 0$ whence $[L] = [i^{ -1 }L] + [pL]$ in $\K_0( \fl\, \sR
)$.  This gives the disjoint union in the lemma which is clearly
finite.
\end{proof}

\begin{Lemma}
\label{lem:pi}
\begin{enumerate}

  \item  If the sequence \eqref{equ:astast} is split exact then $\pi_{
      e,f }$ is surjective. 

\medskip

  \item If $( e,f ) \neq ( 0,[Gc] )$ and $a \rightarrow b \rightarrow c$
  is an AR triangle then $\pi_{ e,f }$ is surjective.

\medskip

  \item If $( e,f ) = ( 0,[Gc] )$ then either $\pi_{ e,f }$ is
  surjective or $X_{ e,f } = \emptyset$.  The former happens if and
  only if the sequence \eqref{equ:astast} is split exact.

\medskip

  \item If $( e,f ) = ( 0,[Gc] )$ then $\Gr_e( Ga ) \times \Gr_f( Gc )
    = \{\, ( 0,Gc ) \,\}$ has only one point.

\medskip

  \item  Each fibre of $\pi_{ e,f }$ is an affine space over $\BC$. 

\end{enumerate}
\end{Lemma}

\begin{proof}
For (i) and (ii) let $( K,M ) \in \Gr_e ( Ga ) \times \Gr_f ( Gc )$
be given.  That is, $K \subseteq Ga$, $M \subseteq Gc$ are subobject
with $[K] = e$ and $[M] = f$.

(i)  When the sequence \eqref{equ:astast} is split exact we set $L = K
\oplus M \subseteq Ga \oplus Gc = Gb$ whence $i^{ -1 }L = K$, $pL = M$
so $\pi_{ e,f }(L) = (K,M)$.

(ii)  Pick $z \in ( \Sigma^{ -1 }\sR ) * \sR$ such that there is an
isomorphism $Gz \stackrel{\sim}{\rightarrow} M$, see Lemma
\ref{lem:5}(i).  Composing it with the inclusion $M \subseteq Gc$
gives a monomorphism $Gz \rightarrow Gc$ which has the form $G( z
\stackrel{\zeta}{\rightarrow} c )$ by Lemma \ref{lem:5}(ii).  Note that $M
= \Image G\zeta$.

First, suppose $e \neq 0$.  In this case, $K \neq 0$ by Equation
\eqref{equ:K}.

By paragraphs \ref{bfhpg:coherents} and \ref{bfhpg:fl} we can pick
$r \in \sR$ such that there is an epimorphism $P_r = G(
\Sigma^{ -1 }r ) \twoheadrightarrow K$.  Composing it with the
inclusion $K \subseteq Ga$ gives a morphism $G( \Sigma^{ -1 }r )
\rightarrow Ga$ which has the form $G( \Sigma^{ -1 }r
\stackrel{\varphi}{\rightarrow} a )$ by Lemma \ref{lem:5}(ii).  Note
that $K = \Image G\varphi$ and that $K \neq 0$ implies $\varphi \neq
0$.

We are assuming that there is an AR triangle $a \rightarrow b
\rightarrow c \stackrel{\gamma}{\rightarrow} \Sigma a$ and since
$\varphi$ and hence $\Sigma \varphi$ are non-zero, $\gamma$ factors as
$c \stackrel{\varepsilon}{\rightarrow} r \stackrel{\Sigma
\varphi}{\rightarrow} \Sigma a$.  We can spin this into the following
commutative diagram where the top row is also a distinguished
triangle.
\[
  \xymatrix @-0.25pc {
    \Sigma^{ -1 }r \ar[r] \ar_{\varphi}[d] & y \ar[r] \ar_{\upsilon}[d] & z \ar^{\varepsilon \zeta }[r] \ar_{\zeta}[d] & r \ar^{ \Sigma \varphi }[d] \\
    a \ar[r] & b \ar[r] & c \ar_{ \gamma }[r] \ar^{\varepsilon}[ur] & \Sigma a
                     }
\]
Applying $G$ gives a commutative diagram with exact rows. 
\[
  \xymatrix @-0.25pc {
    & G( \Sigma^{ -1 }r ) \ar[r] \ar_{G\varphi}[d] & Gy \ar[r] \ar_{G\upsilon}[d] & Gz \ar[r] \ar^{G\zeta}[d] & 0 \\
    0 \ar[r] & Ga \ar_i[r] & Gb \ar_p[r] & Gc \ar[r] & 0
                     }
\]
Set $L = \Image G\upsilon$.  A diagram chase using that $G\zeta$ is a
monomorphism shows $i^{ -1 }L = \Image G\varphi = K$ and $pL = \Image
G\zeta = M$ so $\pi_{ e,f }(L) = (K,M)$.

Secondly, suppose $e = 0$.  We are assuming $( e,f ) \neq ( 0,[Gc] )$
so $f \neq [ Gc ]$ follows.  In this case, $K = 0$ and $M \neq Gc$ by
Equation \eqref{equ:K}. 

Since $M = \Image G\zeta$ the morphism $Gz \stackrel{G\zeta}{\longrightarrow}
Gc$ is not an epimorphism, so $\zeta$ is not a split epimorphism, so $\zeta$
factors as $z \stackrel{\upsilon}{\rightarrow} b \rightarrow c$.  Applying
$G$ gives the following commutative diagram.
\[
  \xymatrix @-0.25pc {
    & & Gz \ar@{=}[r] \ar_{G\upsilon}[d] & Gz \ar^{G\zeta}[d] & \\
    0 \ar[r] & Ga \ar_i[r] & Gb \ar_p[r] & Gc \ar[r] & 0
                     }
\]
Set $L = \Image G\upsilon$.  As above, a diagram chase using that
$G\zeta$ is a monomorphism shows $i^{ -1 }L = 0 = K$ and $pL = \Image
G\zeta = M$ so $\pi_{ e,f }(L) = (K,M)$.

(iii)  When $( e,f ) = ( 0,[Gc] )$, it is clear from part (iv) that
either $\pi_{ e,f }$ is surjective or $X_{ e,f } = \emptyset$.  The
former happens if and only if there is a subobject $L \subseteq Gb$
such that $i^{ -1 }L = 0$ and $pL = Gc$.  This is clearly equivalent
to the existence of a morphism $Gc \stackrel{q}{\rightarrow} Gb$ with
$pq = \id$, that is, equivalent to the short exact sequence
\eqref{equ:astast} being split exact.

(iv)  Follows from Equation \eqref{equ:K}.

(v)  See \cite[lem.\ 3.11]{CC} which is stated for AR sequences, but
has a proof that also works in the present situation.
\end{proof}

\section{Generalised friezes}
\label{sec:friezes}

This section shows Theorem \ref{thm:A} which is a refined version of
Theorem A from the introduction.

\begin{Definition}
\label{def:rho}
For $c \in \sC$ set
\[
  \rho_{ \sR }(c) = \sum_e \chi \big( \Gr_e( Gc ) \big).
\]
Recall that $Gc$ is the $\sR$-module $\sC( -,\Sigma c )|_{ \sR }$ and
$\Gr_e( Gc )$ is the Grassmannian of subobjects $M \subseteq Gc$ with
finite length and $[M] = e$, while $\chi$ is the Euler characteristic
defined by cohomology with compact support, see \cite[p.\ 93]{F}.  The
sum is over $e \in \K_0( \fl\, \sR )$. 

Note that if $Gc = 0$ then $\rho_{ \sR }( c ) = 1$.  However, for other
objects $c$ the formula may not make sense because $Gc$ may have
infinite length, in which case the sum may be infinite.
\end{Definition}

\begin{Definition}
If $\sR = \add\, R$ for a rigid object $R$ then we write
$\rho_R$ instead of $\rho_{ \sR }$; this is the situation from the
introduction. 

For the rest of this section and the next, $\rho_{ \sR }$ is
abbreviated to $\rho$.
\end{Definition}

\begin{Theorem}
\label{thm:A}
\begin{enumerate}

  \item  If $Gc$ is of finite length, then the formula for $\rho( c )$
    makes sense. 

\medskip

  \item If $c_1, c_2 \in \sC$ have $Gc_1$, $Gc_2$ of finite
  length, then $G( c_1 \oplus c_2 )$ has finite length and $\rho( c_1
  \oplus c_2 ) = \rho( c_1 )\rho( c_2 )$.

\medskip

  \item If
\[
  \Delta = \tau c \rightarrow b \rightarrow c
\]
is an AR triangle in $\sC$ and $G( \tau c )$, $Gc$ have finite length,
then so does $Gb$ and
\[
  \rho( \tau c )\rho( c ) - \rho( b )
  = \left\{
      \begin{array}{cl}
        0 & \mbox{ if $G( \Delta )$ is a split short exact sequence, } \\[2mm]
        1 & \mbox{ if $G( \Delta )$ is not a split short exact sequence. }
      \end{array}
    \right.
\]

\end{enumerate}
\end{Theorem}

\begin{proof}
(i)  If $Gc$ has finite length then $\Gr_e( Gc )$ is non-empty only for
finitely many values of $e$, see \cite[par.\ 1.6 and 1.8]{JP}.  Hence
the formula for $\rho$ makes sense.

(iii)  Consider the AR triangle $\Delta$ and suppose that $G( \tau c
)$, $Gc$ have finite length.  The exact sequence $G( \Delta )$ shows
that $Gb$ has finite length.  We now split into cases.

Case (a): $c = \Sigma^{-1}r \in \Sigma^{-1}\ind\, \sR$.  Lemma
\ref{lem:6a}(i) says
\[
  G( \Delta ) = 0 \rightarrow \rad\, P_r \rightarrow P_r;
\]
in particular $G( \Delta )$ is not a split short exact sequence.  We
have 
\begin{align*}
  \rho( c )
  & = \sum_e \chi \big( \Gr_e( P_r ) \big) \\
  & = \chi \big( \Gr_{ [P_r] }( P_r ) \big)
      + \sum_{ e \neq [P_r] } \chi \big( \Gr_e( P_r ) \big) \\
  & = 1
      + \sum_{ e } \chi \big( \Gr_e( \rad\, P_r ) \big) \\
  & = 1 + \rho( b ).
\end{align*}
The penultimate $=$ holds because Equation \eqref{equ:K} implies that
$\Gr_{ [P_r] }( P_r ) = \{\, P_r \,\}$ has only one point, and that
each subobject $M \subseteq P_r$ with $[M] \neq [P_r]$ is proper,
hence contained in $\rad\, P_r$.  Moreover,
\[
  \rho( \tau c ) = 1
\]
since $G( \tau c ) = 0$.  Combining the equations shows
\begin{equation}
\label{equ:1}
  \rho( \tau c )\rho( c ) - \rho( b ) = 1.
\end{equation}

Case (b): $c = r \in \ind\, \sR$.  We can use the dual argument to
Case (a), based on Lemma \ref{lem:6a}(ii).  We get that $G( \Delta )$
is not a split short exact sequence, and Equation
\eqref{equ:1} remains true.

Case (c): $c \not\in \Sigma^{ -1 }( \ind\, \sR ) \cup \ind\, \sR$.  We
will use the machinery of Section \ref{sec:Gr} so set $a \rightarrow b
\rightarrow c$ of Setup \ref{set:Gr} equal to $\Delta = \tau c
\rightarrow b \rightarrow c$.  The requirements of the Setup are
satisfied because $G( \Delta )$ is a short exact sequence by Lemma
\ref{lem:6a}(iii). 

We have
\begin{align*}
  \rho( \tau c )\rho( c )
  & = \sum_{ e,f } \chi \Big( \Gr_e\big( G( \tau c) \big) \Big) \chi \big( \Gr_f( Gc ) \big) \\
  & = \sum_{ e,f } \chi \Big( \Gr_e\big( G( \tau c) \big) \times \Gr_f( Gc ) \Big) \\
  & = \chi \Big( \Gr_0\big( G( \tau c) \big) \times \Gr_{[Gc]}( Gc ) \Big) \\
  & \;\;\;\;\;\;\;\;\; + \sum_{ ( e,f ) \neq ( 0,[Gc] ) } \chi \Big( \Gr_e\big( G( \tau c) \big) \times \Gr_f( Gc ) \Big) \\
  & = \chi \Big( \Gr_0\big( G( \tau c) \big) \times \Gr_{[Gc]}( Gc ) \Big)
    + \sum_{ ( e,f ) \neq ( 0,[Gc] ) } \chi ( X_{ e,f }).
\end{align*}
The second $=$ is by \cite[p.\ 92, item (4)]{F} and the last $=$ is by
\cite[p.\ 93, exercise]{F} and Lemma \ref{lem:pi}(ii)+(v).  On the
other hand,
\[
  \rho( b )
  = \sum_g \chi \big( \Gr_g ( Gb ) \big)
  = \sum_{ e,f } \chi( X_{ e,f } )
  = \chi( X_{ 0,[Gc] } ) 
      + \sum_{ ( e,f ) \neq ( 0,[Gc] ) } \chi ( X_{ e,f }),
\]
where the second $=$ is by \cite[p.\ 92, item (3)]{F} and Lemma
\ref{lem:union}.  It follows that
\[
  \rho( \tau c )\rho( c ) - \rho( b )
  = \chi \Big( \Gr_0\big( G( \tau c) \big) \times \Gr_{[Gc]}( Gc ) \Big)
    - \chi( X_{ 0,[Gc] } )
  = ( \dagger ).
\]
If $G( \Delta )$ is split exact, then $\pi_{ 0,[Gc] }$ is surjective
by Lemma \ref{lem:pi}(i) whence $( \dagger ) = 0$ by \cite[p.\ 93,
exercise]{F} and Lemma \ref{lem:pi}(v).  If $G( \Delta )$ is not split
exact, then Lemma \ref{lem:pi}(iii)+(iv) implies $( \dagger ) = 1-0 =
1$.

(ii)  Suppose that $Gc_1$, $Gc_2$ have finite length.  It is clear that
$G( c_1 \oplus c_2 )$ has finite length.  Set $a \rightarrow b
\rightarrow c$ of Setup \ref{set:Gr} equal to $c_1 \rightarrow c_1
\oplus c_2 \rightarrow c_2$.  A simplified version of the above computation
for Case (c), using part (i) of Lemma \ref{lem:pi} instead of part
(ii), shows $\rho( c_1 \oplus c_2 ) = \rho( c_1 )\rho( c_2 )$.
\end{proof}

\begin{Definition}
\label{def:frieze}
Let $A$ be a commutative ring.  A generalised frieze on $\sC$ with
values in $A$ is a map $\varphi : \obj\, \sC \rightarrow A$ satisfying
\begin{enumerate}

  \item  $\varphi( c_1 \oplus c_2 ) = \varphi( c_1 )\varphi( c_2 )$.

\medskip

  \item  If $\tau c \rightarrow b \rightarrow c$ is an AR triangle in
    $\sC$ then $\varphi( \tau c )\varphi( c ) - \varphi( b )$ equals
    $0$ or $1$. 

\end{enumerate}
\end{Definition}

\begin{Corollary}
\label{cor:rho}
If $Gc$ has finite length for each $c \in \sC$, then $\rho$ is
a generalised frieze with values in $\BZ$. 
\end{Corollary}

\begin{proof}
Immediate from Theorem \ref{thm:A}. 
\end{proof}

\begin{Remark}
Theorem A in the introduction follows from this since it is clear that each $Gc$ 
has finite length when $\sR = \add\, R$ for a rigid object $R$.

However, Theorem \ref{thm:A} is a bit finer because it also deals with
situations where $\rho$ is not defined on every $c \in \sC$.
\end{Remark}

\section{An extension formula}
\label{sec:extension}

This section shows Proposition \ref{pro:extension} which is akin to
the ``exchange relation'' or ``multiplication property'' for cluster
maps, albeit in a special case.  See \cite[introduction]{CK2} and
\cite[introduction]{Palu}.

\begin{Setup}
\label{set:extension}
In this section $\sC$ is assumed to be $2$-Calabi-Yau, that is, its Serre functor is
$S = \Sigma^2$.

Moreover, $m \in \ind\, \sC$ and $r \in \ind \sR$ denote objects satisfying
\[
  \dim_{ \BC } \Ext_{ \sC }^1( r,m )
  = \dim_{ \BC } \Ext_{ \sC }^1( m,r )
  = 1,
\]
and $m \rightarrow a \rightarrow r$ and $r \rightarrow b \rightarrow m$
are the ensuing non-split extensions.
\end{Setup}

\begin{Remark}
\label{rmk:extension}
Being more verbose, we have the following distinguished triangles with
$\delta, \varepsilon \neq 0$. 
\[
\;\;\;\;\;\;
  m
  \stackrel{ \mu }{ \rightarrow } a
  \rightarrow r
  \stackrel{ \delta }{ \rightarrow } \Sigma m
\;\;,\;\;
  r
  \rightarrow b
  \stackrel{ \beta }{ \rightarrow } m
  \stackrel{ \varepsilon }{ \rightarrow } \Sigma r.
\]
Applying $G$ gives exact sequences in $\Mod\, \sR$,
\[
  G( \Sigma^{ -1 }r )
  \stackrel{ G( \Sigma^{-1}\delta ) }{ \longrightarrow } Gm
  \stackrel{ G\mu }{ \longrightarrow } Ga
  \rightarrow 0
\;\;,\;\;
  0
  \rightarrow Gb
  \stackrel{ G\beta }{ \longrightarrow } Gm
  \stackrel{ G\varepsilon }{ \longrightarrow }
  G( \Sigma r ).
\]
\end{Remark}

\begin{Lemma}
\label{lem:dichotomy}
If $M \subseteq Gm$ then either $\Ker G\mu \subseteq M$ or $M
\subseteq \Image G\beta$, but not both.
\end{Lemma}

\begin{proof}
Equivalently, either $\Image G( \Sigma^{-1}\delta ) \subseteq M$ or
$M \subseteq \Ker G\varepsilon$, but not both.

``Not both'': Since $\sC$ is $2$-Calabi-Yau, its AR translation is
$\tau = \Sigma$, so there is an AR triangle $\Sigma r \rightarrow y
\rightarrow r \stackrel{ \sigma }{ \rightarrow } \Sigma^2 r$.  The
morphism $r \stackrel{ \delta }{ \rightarrow } \Sigma m$ is non-zero,
so $\sigma$ factors as $r \stackrel{ \delta }{ \rightarrow } \Sigma m
\stackrel{ \psi }{ \rightarrow } \Sigma^2 r$.  Since $\psi\delta =
\sigma \neq 0$, we have $\psi \neq 0$.  It therefore follows from
$\dim_{ \BC } \sC( \Sigma m,\Sigma^2 r ) = \dim_{ \BC } \sC( m,\Sigma
r ) = 1$ that $\Sigma m \stackrel{ \Sigma \varepsilon}{
\longrightarrow } \Sigma^2 r$ is a non-zero scalar multiple of $\psi$,
whence $\psi\delta \neq 0$ implies $\Sigma( \varepsilon )\delta \neq
0$.  Hence $G( \varepsilon \Sigma^{ -1 }\delta ) \neq 0$, because this
morphism is
\[
  \xymatrix @-0.25pc {
    \sC( - , r ) |_{ \sR }
    \ar^{ ( \Sigma( \varepsilon )\delta )_{ * } }[rr] & &
    \sC( - , \Sigma^2 r ) |_{ \sR }.
            }
\]

Now suppose $\Image G( \Sigma^{ -1 }\delta ) \subseteq M$.  Applying
$G\varepsilon$ gives $\Image G( \varepsilon \Sigma^{ -1 }\delta )
\subseteq ( G\varepsilon )M$.  By what we showed above, this implies
$( G\varepsilon )M \neq 0$, that is $M \not\subseteq \Ker
G\varepsilon$ as claimed. 

``Either/or'': Suppose that $M \not\subseteq \Ker G\varepsilon$.
Since $G\varepsilon$ is
\[
  \xymatrix @-0.75pc {
    \sC( - , \Sigma m ) |_{ \sR }
    \ar^{ ( \Sigma \varepsilon )_{ * } }[rr] & &
    \sC( - , \Sigma^2 r ) |_{ \sR }
            }
\]
this means there exist $r' \in \ind\, \sR$ and a morphism $r'
\stackrel{ \rho' }{ \rightarrow } \Sigma m$ in $M( r' )$ such that
the composition $r' \stackrel{ \rho' }{ \rightarrow } \Sigma m
\stackrel{ \Sigma \varepsilon }{ \longrightarrow } \Sigma^2 r$ is
non-zero.  Hence the map $\sC( \Sigma m,\Sigma^2 r ) \stackrel{
  \rho'^{*} }{ \rightarrow } \sC( r',\Sigma^2 r )$ is non-zero, whence
the lower horizontal map is non-zero in the following commutative
square which exists by Serre duality.
\[
  \xymatrix @! @C=2.5pc @R=-3.0pc {
    \sC( r,r' ) \ar^{ \rho'_{*} }[r] \ar_{ \cong }[d]
    & \sC( r,\Sigma m ) \ar^{ \cong }[d] \\
    \dual\!\sC( r',\Sigma^2 r ) \ar_-{ \dual( \rho'^{*} ) }[r]
    & \dual\!\sC( \Sigma m,\Sigma^2 r )
               }
\]
It follows that the upper horizontal map is non-zero, so surjective
since $\dim_{ \BC } \sC( r,\Sigma m ) = 1$ by assumption.  Hence
$r \stackrel{ \delta }{ \rightarrow } \Sigma m$ factors as $r
\rightarrow r' \stackrel{ \rho' }{ \rightarrow } \Sigma m$.

However, for $r'' \in \sR$ each element of $ \big(
\Image G( \Sigma^{ -1 }\delta ) \big)( r'' )$ is a composition $r''
\rightarrow r \stackrel{ \delta }{ \rightarrow } \Sigma m$.  By what
we have showed, such a composition can also be written as a
composition $r'' \rightarrow r' \stackrel{ \rho' }{ \rightarrow }
\Sigma m$ so is in $M( r'' )$.  Hence $\Image G( \Sigma^{ -1 }\delta
) \subseteq M$ as desired.
\end{proof}

\begin{Proposition}
\label{pro:extension}
In the situation of Setup \ref{set:extension}, if $Gm$ has finite
length then so do $Ga$ and $Gb$, and 
\[
  \rho( m ) = \rho( a ) + \rho( b ).
\]
\end{Proposition}

\begin{proof}
The claim about lengths follows from the exact sequences in Remark
\ref{rmk:extension}.  

When $Gm$ has finite length there are injections
\[
  \xymatrix @R=3pt @C=1.0pc {
    & \Gr_{ e - [ \Ker G\mu ] }( Ga ) \ar@{^{(}->}[r] & \Gr_e( Gm ) & \Gr_e( Gb ) \ar@{_{(}->}[l] & \!\!\!\!\!\!\!\!\lefteqn{,} \\
    & K \ar@{|->}[r] & ( G\mu )^{ -1 } K & & \!\!\!\!\!\!\!\!\lefteqn{,} \\
    & & ( G\beta ) L & L \ar@{|->}[l] & \!\!\!\!\!\!\!\!\lefteqn{.}
            }
\]
The images are constructible by paragraph
\ref{bfhpg:constructible_maps} and they are disjoint with union equal
to $\Gr_e( Gm )$ by Lemma \ref{lem:dichotomy}, whence
\[
  \chi \big( \Gr_e( Gm )  \big)
  = \chi \big( \Gr_{ e - [ \Ker G\mu ] }( Ga ) \big)
  + \chi \big( \Gr_e( Gb ) \big)
\]
by \cite[p.\ 92, item (3)]{F}.  Summing over $e \in \K_0( \fl\, \sR )$
proves the proposition.
\end{proof}

\begin{Remark}
\label{rmk:extension2}
Since $Gr = 0$ we have $\rho( r ) = 1$, so the proposition can also be
written
\begin{equation}
\label{equ:extension2}
  \rho( m )\rho( r ) = \rho( a ) + \rho( b ).
\end{equation}
This makes it clearer that it is akin to the ``exchange
relation'' or ``multiplication property'' for cluster characters, see
\cite[introduction]{CK2} and \cite[introduction]{Palu}.

If $r \in \ind\, \sC$ then Equation \eqref{equ:extension2} holds for
cluster characters but may fail for $\rho$, see Remark
\ref{rmk:extension3}. 
\end{Remark}

\section{The generalised friezes of \cite{BHJ}}
\label{sec:BHJ}

This section shows Theorem \ref{thm:B} which is a reformulation of
Theorem B in the introduction.

\begin{Setup}
\label{set:Type_A}
In this section, $n \geq 3$ is an integer, $\sC = \sC( A_n )$ is the
cluster category of type $A_n$, see \cite{BMRRT} and \cite{CCS}, and
$R$ is a rigid object of $\sC$ without repeated indecomposable
summands.  We set $\sR = \add\, R$, see paragraph \ref{bfhpg:rigid}.
\end{Setup}

\begin{bfhpg}
[Coordinates and diagonals]
\label{bfhpg:Type_A}
It is clear that $Gc$ has finite length for each $c \in \sC$, and
well known that $\sC$ and $\sR$ satisfy the conditions of Setups
\ref{set:blanket} and \ref{set:extension}, so the results of Sections
\ref{sec:friezes} and \ref{sec:extension} apply.

The following properties were shown in \cite{CCS}: the AR quiver of
$\sC$ is $\BZ A_n$ modulo a certain glide reflection.  There is a
coordinate system on the AR quiver of $\sC$, part of which is shown in
Figure \ref{fig:coordinates}.  It is continued with the stipulation
that the order of the coordinates does not matter and individual
coordinates are taken modulo $n+3$; this emulates the action of the
glide reflection.
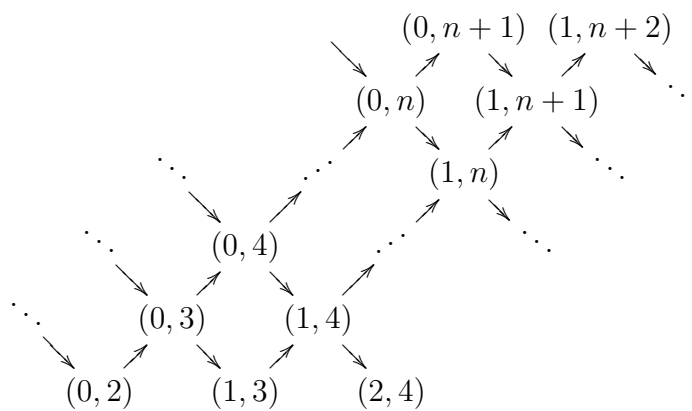
\begin{figure}
\[
  \vcenter{
  \xymatrix @+0.3pc @!0 {
& & & & {\;} \ar[dr] & & (0,n+1) \ar[dr] & & (1,n+2) \ar[dr] & \\
& & & & & (0,n) \ar[dr] \ar[ur] & & (1,n+1)\ar[dr] \ar[ur] & & \ddots \\
& & \ddots \ar[dr] & & \adots \ar[ur] & & (1,n) \ar[dr] \ar[ur] & & \ddots & \\
& \ddots \ar[dr] & & (0,4) \ar[dr] \ar[ur] & & \adots \ar[ur] & & \ddots & & \\
\ddots \ar[dr] & & (0,3) \ar[dr] \ar[ur] & & (1,4) \ar[dr] \ar[ur] & & & & & \\
& (0,2) \ar[ur] & & (1,3) \ar[ur] & & (2,4) & & & & \\
            }
          }
\]
\caption{The coordinate system on the AR quiver of $\sC( A_n )$.}
\label{fig:coordinates}
\end{figure}
We think of the coordinate pair $( i,j )$ as the diagonal connecting
vertices $i$ and $j$ in a regular $( n+3 )$-gon $P$ with vertex set
$\{\, 0, \ldots, n+2 \,\}$.  This identifies the indecomposable
objects of $\sC$ with the diagonals of $P$.  The identification has
the property that if $M, S \in \ind\, \sC$ then
\begin{equation}
\label{equ:cross}
  \dim_{ \BC } \Ext^1_{ \sC }( M,S )
  = \left\{
      \begin{array}{cl}
        1 & \mbox{ if $M$ and $S$ cross, } \\[1mm]
        0 & \mbox{ if not. }
      \end{array}
    \right.
\end{equation}
In particular, the indecomposable summands of the rigid object $R$ is
a set of pairwise non-crossing diagonals of $P$, that is, a polygon
dissection of $P$ which will also be denoted by $R$.
\end{bfhpg}

\begin{bfhpg}
[The generalised friezes of \cite{BHJ}]
\label{bfhpg:BHJ}
Let us recall the algorithm of \cite[sec.\ 3]{BHJ} which uses the
polygon dissection $R$ of the $( n+3 )$-gon $P$ to define a
generalised frieze on $\sC = \sC( A_n )$.  Note that in \cite{BHJ}
the polygon dissection was assumed to be a higher angulation, but this
restriction is unnecessary.

Define non-negative integers $m_R( i,j )$, indexed by vertices $i,j$
of $P$, by the following inductive procedure:

Let $i$ be fixed.  Set $m_R( i,i ) = 0$.  The polygon dissection $R$
splits $P$ into smaller polygonal pieces.  If $\alpha$ is a piece
containing $i$, and $j$ is another vertex of $\alpha$, then set $m_R(
i,j ) = 1$.  If $\alpha$ is an piece not containing $i$, then we can
assume that there is a piece $\alpha'$ sharing an edge $( k,\ell )$
with $\alpha$, such that $m_R( i,j )$ has already been defined for
the vertices $j$ of $\alpha'$.  Set
\begin{equation}
\label{equ:BHJ}
  m_R( i,j ) = m_R( i,k ) + m_R( i,\ell )
\end{equation}
for each vertex $j \neq k,\ell$ of $\alpha$.  Note that $( k,\ell )$
is a diagonal in $R$, that is, an indecomposable summand of $R$.

It was proved in \cite[thm.\ 3.3]{BHJ} that $m_R( i,j ) = m_R( j,i )$,
so $m_R$ can be viewed as being defined on the diagonals of $P$, that
is, on the indecomposable objects of $\sC$.  It is extended to all
objects by the rule $m_R( c_1 \oplus c_2 ) = m_R( c_1 )m_R( c_2 )$.

Moreover, the AR triangles in $\sC$ have the form
\[
  ( i-1,j-1 )
  \rightarrow ( i-1,j ) \oplus ( i,j-1 )
  \rightarrow ( i,j )
\]
where $( i-1,j )$ and $( i,j-1 )$ have to be interpreted as $0$ if
their coordinates are neighbouring vertices of $P$, and it was proved
in \cite[thm.\ 5.1]{BHJ} that each difference
\begin{equation}
\label{equ:difference}
  m_R( i-1,j-1 )m_R( i,j ) - m_R( i-1,j )m_R( i,j-1 ) 
\end{equation}
equals $0$ or $1$.

Hence $m_R$ is a generalised frieze on $\sC$. 
\end{bfhpg}

\begin{Theorem}
\label{thm:B}
Consider the situation of Setup \ref{set:Type_A}.  The rigid object
$R$ gives a polygon dissection of the $( n+3 )$-gon $P$, see paragraph
\ref{bfhpg:Type_A}, and the dissection gives a generalised frieze $m_R$ on
$\sC$, see paragraph \ref{bfhpg:BHJ}.

The rigid object $R$ also gives a generalised friese $\rho_R$ on
$\sC$, see Definition \ref{def:rho} and Corollary \ref{cor:rho}.  

These generalised friezes agree, that is, $m_R = \rho_R$.
\end{Theorem}

\begin{proof}
Since $m_R( c_1 \oplus c_2 ) = m_R( c_1 )m_R( c_2 )$ by definition and
$\rho_R( c_1 \oplus c_2 ) = \rho_R( c_1 )\rho_R( c_2 )$ by Theorem
\ref{thm:A}(ii), it is enough to let $i$ be a fixed vertex of $P$ and
show
\begin{equation}
\label{equ:B}
  m_R( i,j ) = \rho_R \big( ( i,j ) \big )
\end{equation}
for each vertex $j$ of $P$, and we do so inductively:

The polygon dissection $R$ splits $P$ into smaller polygonal pieces.
If $\alpha$ is a piece containing $i$, and $j$ is another vertex of
$\alpha$, then by definition $m_R( i,j ) = 1$.  The diagonal $( i,j )$
crosses none of the diagonals in $R$, so $\Ext^1_{ \sC }\big( R,( i,j
) \big) = 0$ by Equation \eqref{equ:cross}.  That is, $G \big( ( i,j )
\big) = 0$ so $\rho_R \big( ( i,j ) \big) = 1$, verifying Equation
\eqref{equ:B}.

If $\alpha$ is a piece not containing $i$, then we can assume that
there is a piece $\alpha'$ sharing an edge $S = ( k,\ell )$ with
$\alpha$, such that if $j$ is a vertex of $\alpha'$ then Equation
\eqref{equ:B} has already been verified, and such that if $j \neq k,
\ell$ is a vertex of $\alpha$ then $M = ( i,j )$ crosses $S$.  For
such a $j$,
\[
  \dim_{ \BC } \Ext^1_{ \sC }( M,S )
  = \dim_{ \BC } \Ext^1_{ \sC }( S,M )
  = 1
\]
by Equation \eqref{equ:cross}, and there are non-split extensions
\[
  M \rightarrow A \oplus A' \rightarrow S
\;\;,\;\;
  S \rightarrow B \oplus B' \rightarrow M
\]
in $\sC$ where $A,A',B,B' \in \ind\, \sC$ are the diagonals in Figure 
\ref{fig:1}.  
\begin{figure}
\[
  \begin{tikzpicture}[auto]

    \node[name=s, shape=regular polygon, regular polygon sides=25, minimum size=6.5cm, draw] {}; 
    \draw[thick] (s.corner 16) to node[near end] {$S$} (s.corner 25);
    \draw[thick] (s.corner 9) to node {$M$} (s.corner 21);
    \draw[thick] (s.corner 9) to node {$A'$} (s.corner 25);
    \draw[thick] (s.corner 21) to node {$B$} (s.corner 25);
    \draw[thick] (s.corner 16) to node {$B'$} (s.corner 9);
    \draw[thick] (s.corner 16) to node {$A$} (s.corner 21);

    \draw[shift=(s.corner 9)] node[left] {$j$};
    \draw[shift=(s.corner 25)] node[above] {$k$};
    \draw[shift=(s.corner 16)] node[below] {$\ell$};
    \draw[shift=(s.corner 21)] node[right] {$i$};

  \end{tikzpicture} 
\]
\caption{There are non-split extensions $M \rightarrow A \oplus A'
  \rightarrow S$ and $S \rightarrow B \oplus B' \rightarrow M$ in
  $\sC( A_n )$.}
\label{fig:1}
\end{figure}
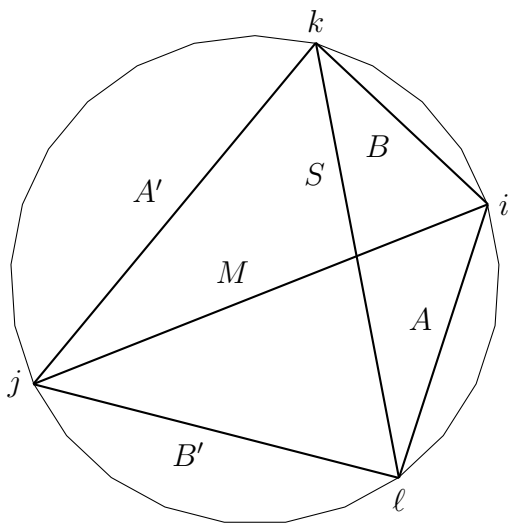
Note that if one or more of $A,A',B,B'$ are edges of $P$, then they
must be interpreted as zero objects, and that $S$ is a diagonal in
$R$, that is, an indecomposable summand of $R$.

Combining Proposition \ref{pro:extension} and Theorem \ref{thm:A}(ii)
gives
\begin{equation}
\label{equ:rho}
  \rho_R( M )  = \rho_R( A )\rho_R( A' ) + \rho_R( B )\rho_R( B' ).
\end{equation}
Since $j$, $k$, $\ell$ are vertices of $\alpha$, the diagonals $A' = (
j,k )$ and $B' = ( j,\ell )$ cross none of the diagonals in $R$, so
$GA' = GB' = 0$ by Equation \eqref{equ:cross}, and hence $\rho_R( A' )
= \rho_R( B' ) = 1$.  Equation \eqref{equ:rho} therefore reads
$\rho_R( M ) = \rho_R( A ) + \rho_R( B )$, giving the first of the
following equalities.
\[
  \rho_R \big( ( i,j ) \big)
  = \rho_R \big( ( i,\ell ) \big) + \rho_R \big( ( i,k ) \big)
  = m_R( i,\ell ) + m_R( i,k )
  = m_R( i,j )
\]
The second equality is by assumption since $k,\ell$ are vertices of
$\alpha'$, and the third equality is Equation \eqref{equ:BHJ}.

This shows Equation \eqref{equ:B} for the vertices $j$ of $\alpha$,
completing the induction.
\end{proof}

\begin{Remark}
\label{rmk:extension3}
Consider the situation of Setup \ref{set:extension}.  Remark
\ref{rmk:extension2} proved Equation \eqref{equ:extension2} for $r \in
\ind\, \sR$.

The remark claimed that if $r \in \ind\, \sC$ then Equation
\eqref{equ:extension2} may fail.  We can now prove this: if it did
always hold, then for $\sC = \sC( A_n )$ we could let the extensions
in Setup \ref{set:extension} be
\[
  ( i-1,j-1 )
  \rightarrow ( i-1,j ) \oplus ( i,j-1 )
  \rightarrow ( i,j )
\;\;,\;\;
  ( i,j )
  \rightarrow 0
  \rightarrow ( i-1,j-1 )
\]
where the first is the AR triangle ending in $( i,j )$ and the second
has connecting morphism equal to the identity on $( i-1,j-1 )$.  Then
Equation \eqref{equ:extension2} would give
\[
  \rho_R \big( ( i-1,j-1 ) \big) \rho_R \big( ( i,j ) \big)
  = \rho_R \big( ( i-1,j ) \big) \rho_R \big( ( i,j-1 ) \big) + 1,
\]
and Theorem \ref{thm:B} would imply that the difference
\eqref{equ:difference} was always $1$.  That is false, however; see
\cite[thm.\ 5.1(c)]{BHJ}.
\end{Remark}

\medskip
\noindent
{\bf Acknowledgement.}
This work grew out of \cite{BHJ}.  We are grateful to Christine
Bessenrodt for the fruitful collaboration on that paper.

We are grateful for very useful input from Yann Palu and Pierre-Guy
Plamondon.  In particular, Pierre-Guy Plamondon visited Newcastle in
March 2014 and his comments to the second author led to significant
improvements of the paper.

We thank the referee for a number of useful comments.

Part of this work was carried out while Peter J\o rgensen was visiting
Hannover.  He thanks Christine Bessenrodt, Thorsten Holm, and the
Institut f\"{u}r Algebra, Zahlentheorie und Diskrete Mathematik at the
Leibniz Universit\"{a}t for their hospitality.  He also gratefully
acknowledges financial support from Thorsten Holm's grant HO 1880/5-1,
which is part of the research priority programme SPP 1388 {\em
  Darstellungstheorie} of the Deutsche Forschungsgemeinschaft (DFG).

\end{document}